\documentclass[11pt,a4paper]{article}
\usepackage[margin=26mm]{geometry}
\usepackage[T1]{fontenc}
\usepackage{lmodern,microtype,amsmath,amssymb,amsthm,mathtools}
\usepackage{enumitem,needspace,booktabs,xcolor,hyperref}
\definecolor{ink}{HTML}{16364A}
\hypersetup{colorlinks=true,linkcolor=ink,urlcolor=ink,citecolor=ink,
pdfauthor={Shaozhen Cao, Zhuoni Chi, and Ping Nie},
pdftitle={Density of Forces Producing Navier--Stokes Blowup}}
\newtheorem{theorem}{Theorem}[section]
\newtheorem{proposition}[theorem]{Proposition}
\newtheorem{lemma}[theorem]{Lemma}
\newtheorem{corollary}[theorem]{Corollary}
\theoremstyle{definition}
\theoremstyle{remark}\newtheorem{remark}[theorem]{Remark}
\newcommand{\R}{\mathbb R}\newcommand{\TT}{\mathbb T^3}
\newcommand{\FF}{\mathcal F}\newcommand{\XX}{\mathcal X}
\newcommand{\BB}{\mathcal B}\newcommand{\PP}{\mathbb P}
\newcommand{\dd}{\,\mathrm d}\newcommand{\eps}{\varepsilon}
\newcommand{\norm}[1]{\left\lVert#1\right\rVert}
\newcommand{\ip}[2]{\left\langle#1,#2\right\rangle}
\DeclareMathOperator{\supp}{supp}
\setlist[enumerate]{label=(\roman*),itemsep=2pt,topsep=4pt}
\AtBeginEnvironment{theorem}{\samepage}
\AtBeginEnvironment{proposition}{\samepage}
\AtBeginEnvironment{lemma}{\samepage}
\AtBeginEnvironment{corollary}{\samepage}
\title{Density of Forces Producing\\ Navier--Stokes Blowup}
\author{Shaozhen Cao\\\texttt{shaozhencao@zju.edu.cn}\and
Zhuoni Chi\\\texttt{znchi@zju.edu.cn}\and
Ping Nie\\\texttt{ping.nie@pku.edu.cn}}
\date{}
\begin{document}
\maketitle
\begin{abstract}
We study the density of smooth external forces for which the three-dimensional
Navier--Stokes equations lose classical regularity by a prescribed time $T>0$.
Starting from the compact forced blowup solution of OpenAI, we construct a
blowup solution near every given smooth solution while preserving its
initial velocity. A smooth cutoff of a local vector potential makes the
given velocity vanish near the support of a rescaled blowup solution.
The two velocities then have no nonlinear interaction, and the resulting
force remains smooth through the blowup time. For fixed viscosity and zero
initial velocity, such forces are dense in the relative $L^1_tH^s_x$ topology
on both $\TT$ and $\R^3$ exactly when $s<1/2$.
\end{abstract}

\section{Introduction}\label{sec:intro}
The three-dimensional incompressible Navier--Stokes equations with
viscosity $\nu>0$ and external force $f$ are
\begin{equation}\label{eq:NS}
 \partial_tu+(u\cdot\nabla)u-\nu\Delta u+\nabla p=f,
 \qquad \nabla\cdot u=0,\qquad u(\cdot,0)=a.
\end{equation}
We work on $\R^3$ and on the unit torus $\TT=\R^3/\mathbb Z^3$.
For fixed $a$ and $T>0$, we ask in which force norms the smooth forces
producing classical breakdown by $T$ are dense.

The geometric argument is a divergence-free version of the smooth Urysohn
construction: make a given smooth velocity vanish near a small compact set,
leave it unchanged outside a slightly larger set, and place a compact blowup
solution in the resulting region. Smooth cutoffs are supplied by the classical
smooth Urysohn lemma~\cite[Corollary~1.3.4]{petersen}. To preserve
incompressibility, apply the cutoff to a local vector potential rather than
to the velocity itself. The remaining equation error is absorbed into the
external force. Our blowup building block is the following result of
OpenAI~\cite[Theorem~1.1]{openai}.

\Needspace{12\baselineskip}
\begin{theorem}\label{thm:packet}
For every $\nu>0$ there exist a force
$f\in C_c^\infty(\R^3\times(0,\infty);\R^3)$, a compact set
$K\subset\R^3$, and smooth velocity and pressure fields $u,p$ on
$\R^3\times[0,1)$ satisfying
\begin{equation}\label{eq:packet}
 \partial_tu+(u\cdot\nabla)u-\nu\Delta u+\nabla p=f,
 \qquad \nabla\cdot u=0,\qquad u(\cdot,0)=0,
\end{equation}
such that $\supp u(\cdot,t)\cup\supp p(\cdot,t)\subset K$ for every
$0\le t<1$,
\begin{equation}\label{eq:packetblowup}
 \sup_{0\le t<1}\norm{u(t)}_{L^2(\R^3)}<\infty,
 \qquad \limsup_{t\uparrow1}\norm{u(t)}_{L^\infty(\R^3)}=\infty.
\end{equation}
Consequently, there is no smooth solution $(u,P)$ on
$\R^3\times[0,\infty)$ with the same force and initial datum whose
kinetic energy is uniformly bounded
$\sup_{t\ge0}\frac12\int_{\R^3}|u(x,t)|^2\dd x<\infty$.
\end{theorem}

Fix one such building block $(U,P,F)$~\cite{openai,lean}.
The local gluing strategy parallels
Enciso, Pe\~nafiel-Tom\'as and Peralta-Salas~\cite[Theorem~1.5,
Lemma~2.11, Sections~11.1 and~11.3]{eppt2025}: first modify a given smooth
solution locally, then glue in a blowup building block. Their construction
produces weak unforced Euler solutions by removing Reynolds stress with
convex integration. Here the force is allowed to change, and the glued
solution is classical before its blowup time.

More explicitly, for a given smooth velocity $v=\nabla\times A$ locally,
choose smooth cutoffs and set
\[
 w_\eps=-\nabla\times(\eta_\eps\theta_\eps A),\qquad
 u_\eps=v+w_\eps+U_\eps.
\]
The corrected velocity $v+w_\eps$ vanishes near the support of the rescaled
building block $U_\eps$. Hence both mixed nonlinear terms vanish. The force
is the sum of the given force, the rescaled force, and a smooth cutoff
correction. Plancherel's theorem and the Fourier dilation formula, followed
by the time change $t=t_\eps+\eps^2\tau$, give the force factor
$\eps^{2/q-3/2-s}$ in $L^q_t\dot H^s_x$, wherever finite.
For $0\le s\le1$, the corresponding fractional estimate also follows from
the classical Sobolev interpolation inequality
\[
 \norm{z}_{\dot H^s}\le\norm{z}_2^{1-s}\norm{\nabla z}_2^s,
\]
which is H\"older's inequality on the Fourier side; see
\cite[Chapter~4, Sections~1--2]{taylor2011}. The cutoff correction has one
additional power of $\eps$. Negative orders require the separate
low-frequency estimate in Section~\ref{sec:whole}.

The resulting sharp thresholds for zero initial velocity are
\begin{center}
\begin{tabular}{lll}
\toprule Domain & Force topology & Density holds exactly when\\
\midrule
$\TT$ & $L^1_tH^s_x$ & $s<1/2$\\
$\R^3$ & $L^1_tH^s_x$ & $s<1/2$\\
$\R^3$ & $L^2_tH^s_x$ & $s<-1/2$\\
\bottomrule
\end{tabular}
\end{center}
Below threshold, density holds for every fixed admissible smooth initial
velocity. Near a solution smooth through $T$, the approximation preserves
its initial velocity and earlier history, blows up exactly at $T$, and
converges in energy and dissipation. For a general given force the conclusion
is breakdown by $T$, since a solution that already breaks down earlier needs
no modification. Non-density at and above threshold follows from the forced
critical regularity estimates, not from scaling alone.

For comparison, Hofmanov\'a, Zhu and Zhu~\cite[Theorem~1.4, Corollary~4.6,
Theorem~5.1]{hzz2024} obtain force-density results for nonunique Leray--Hopf
solutions. Our property is classical breakdown under a smooth force.
Classical local theory and the critical velocity framework are supplied by
Tao~\cite{tao2013,tao2018}. Section~\ref{sec:preliminaries} fixes conventions and collects
the standard estimates used below; Sections~\ref{sec:torus}
and~\ref{sec:whole} give the periodic and whole-space arguments.

\section{Definitions and preliminary results}\label{sec:preliminaries}
\subsection{Norm conventions}
We use the unitary angular Fourier transform on $\R^3$ and Fourier series
$\widehat z(k)=\int_{\TT}z(x)e^{-2\pi i k\cdot x}\dd x$ on the unit torus.
Thus
\[
 \norm{z}_{H^s(\R^3)}^2=\int(1+|\xi|^2)^s|\widehat z(\xi)|^2\dd\xi,
 \qquad
 \norm{z}_{H^s(\TT)}^2=\sum_k(1+4\pi^2|k|^2)^s|\widehat z(k)|^2.
\]
These are the distributional Sobolev spaces of
Taylor~\cite[Chapters~3--4]{taylor2011}; vector and tensor norms sum the
squared component norms. The homogeneous norms replace these weights by
$|\xi|^{2s}$ and $(2\pi|k|)^{2s}$, respectively, with zero mean on the
torus. Tao's unnormalized Fourier convention gives the same Euclidean norms
after its Plancherel factor~\cite[Appendix~A]{taodispersive}.

Time norms are Bochner norms of strongly measurable paths, with equality
almost everywhere understood~\cite[Definitions~6.14, 6.27]{hunterpde}:
\begin{equation}\label{eq:time-norms}
 \norm{z}_{L^q(I;X)}=\left(\int_I\norm{z(t)}_X^q\dd t\right)^{1/q},
 \qquad \norm{z}_{L^\infty(I;X)}=\operatorname*{ess\,sup}_{t\in I}\norm{z(t)}_X.
\end{equation}
Force norms use $(0,\infty)$ unless stated otherwise, and give each smooth
force class its relative norm topology. On $(0,T)$ we write
\begin{equation}\label{eq:Enorm}
 \norm{z}_{E_T}=\norm{z}_{L^\infty_tL^2_x}+\norm{\nabla z}_{L^2_tL^2_x}.
\end{equation}
This is equivalent to the sum norm on $L^\infty_tL^2_x\cap L^2_tH^1_x$;
it assigns no value at the terminal time.

\subsection{Smooth data and solution classes}
All fields are real, and forces need not be divergence free. We write
$L^2_\sigma(D)$ for the divergence-free subspace of $L^2(D;\R^3)$.
The allowed initial velocities and forces are
\begin{align}
 \XX=\mathcal X_{\TT}&=C^\infty_{\mathrm{div}}(\TT;\R^3),&
 \FF=\mathcal F_{\TT}&=C_c^\infty(\TT\times(0,\infty);\R^3),
 \label{eq:inputspaces}\\
 \mathcal X_{\R}&=H^\infty(\R^3;\R^3)\cap L^2_\sigma(\R^3),
 \label{eq:Rinitial}
\end{align}
where $H^\infty=\bigcap_{m=0}^\infty H^m$ has its usual Fr\'echet topology.
On the whole space set
\begin{equation}\label{eq:Rclasses}
 \mathcal F_{\R}=\left\{f\in C^\infty([0,\infty);H^\infty):
 \norm{f}_{L^1_tH^m_x}+\norm{f}_{L^2_tH^m_x}<\infty
 \text{ for every integer }m\ge0\right\}.
\end{equation}
Smoothness into $H^\infty$ means smoothness into each $H^m$, with
one-sided time derivatives at zero. Thus periodic forces vanish near zero
and outside a compact time interval; whole-space forces may be nonzero at
zero and need not have compact support. Both classes consist of forces
defined through any possible blowup time of the velocity.

On the torus, velocity and pressure are periodic and $\int_{\TT}p=0$.
On $\R^3$, a classical velocity belongs to $C([0,S];H^m)$ for every
integer $m\ge0$ on each compact interval of its lifespan. Its pressure
gradient is specified below, and the scalar pressure is determined up to
a function of time. In either case local uniqueness defines a maximal
classical lifespan, denoted by $T^\nu_{\max}(a,f)$ on the torus and
$T^\nu_{\max,\R}(a,f)$ on the whole space. A given smooth solution is
\emph{smooth through $T$} if it extends smoothly to $[0,T+\delta]$
for some $\delta>0$ in the relevant class.

For fixed $\nu,T>0$, define
\begin{align}
 \BB_{\nu,a,T}&=\{f\in\FF:T^\nu_{\max}(a,f)\le T\},&
 \BB^0_{\nu,T}&=\BB_{\nu,0,T},\label{eq:singularforces}\\
 \mathcal B^\R_{\nu,a,T}
 &=\{f\in\mathcal F_{\R}:T^\nu_{\max,\R}(a,f)\le T\}.
 \label{eq:Rsingularforces}
\end{align}
We call these the forces producing classical breakdown by $T$, using
``breakdown for failure of classical continuation, as in the Clay
problem statement~\cite[alternatives (C) and (D)]{clay}.
The glued solutions satisfy the more specific condition of unbounded
speed at their terminal time.

\subsection{Homogeneous spaces and Fourier multipliers}
Write $\Lambda=(-\Delta)^{1/2}$ and $J=(I-\Delta)^{1/2}$, with the
Fourier conventions of Section~\ref{sec:intro}. On mean-zero periodic
fields, the homogeneous and inhomogeneous Sobolev norms are equivalent
at each fixed order. For $0<a<3/2$, the homogeneous completion is represented by
$\widehat z=|\xi|^{-a}G$, $G\in L^2$, with its unique
$L^{6/(3-2a)}$ representative supplied by Sobolev embedding.
The distributional pairing is well defined since $|\xi|^{-a}$ is locally
square integrable and Schwartz tests decay rapidly.
For the completed energy-force space we fix
\begin{equation}\label{eq:homogeneous-realization}
 \dot H^{-1}(\R^3)=\{h\in\mathcal S':\widehat h=F
 \text{ is a measurable function and }|\xi|^{-1}F\in L^2\}.
\end{equation}
Functions representing distributions are identified almost everywhere.
The maps $h\mapsto|\xi|^{-1}\widehat h$ and
$h\mapsto\langle\xi\rangle^s\widehat h$ identify these spaces with $L^2$;
the inverse Fourier data define tempered distributions. Thus they are
separable Hilbert spaces, with no polynomial ambiguity. Real vector fields
correspond to the conjugate-symmetric Fourier data.

\subsection{Standard Sobolev estimates}\label{sec:sobolev-estimates}
We use the tame product estimate and $H^2\hookrightarrow L^\infty$:
\begin{equation}\label{eq:Rproduct}
 \norm{vw}_{H^m}\le C_m(\norm{v}_{H^2}\norm{w}_{H^m}
              +\norm{w}_{H^2}\norm{v}_{H^m}),\qquad
 \norm{v}_\infty\le C\norm{v}_{H^2},\quad m\ge2.
\end{equation}
See Tao~\cite[Appendix~A, (A.13), Lemma~A.8]{taodispersive}; the periodic
version follows by the same Fourier convolution estimate, with
$\|\widehat v\|_{\ell^1}\le C\|v\|_{H^2}$. The critical embeddings used below are
\begin{equation}\label{eq:embeddings}
 \norm{v}_3\le C\norm{v}_{\dot H^{1/2}},\qquad
 \norm{\nabla v}_3+\norm{\Lambda v}_3\le C\norm{v}_{\dot H^{3/2}},\qquad
 \norm{\nabla v}_6\le C\norm{\Delta v}_2.
\end{equation}
They follow from the homogeneous Sobolev inequality
\cite[Appendix~A, (A.11)]{taodispersive}, applied also to derivatives.
On the torus use the compact-manifold embedding
\cite[Chapter~13, Proposition~6.4 and preceding discussion]{taylor2011iii}
and the spectral gap for mean-zero fields. These statements use the
homogeneous representatives fixed above; no embedding of $\dot H^{3/2}$
into $L^\infty$ is used.

\subsection{Projection, pressure, and local existence}\label{sec:analytic}
The Leray projection $\PP$ has Fourier symbol
$I-\xi\otimes\xi/|\xi|^2$ on $\R^3\setminus\{0\}$ and the same
symbol at $k\ne0$ on the torus. At the periodic zero mode it is the
identity. It is bounded on every $H^s$ and commutes with derivatives and
the heat semigroup. The projected equation is
\begin{equation}\label{eq:projected}
 \partial_tu-\nu\Delta u=-\PP\nabla\cdot(u\otimes u)+\PP f.
\end{equation}
On the torus, pressure is recovered from
\[
 \Delta p=\nabla\cdot f-\sum_{i,j}\partial_i\partial_j(u_i u_j),
 \qquad \int_{\TT}p=0.
\]
The right side has zero mean. On $\R^3$ we require
\begin{equation}\label{eq:Rpressure}
 \nabla p=(I-\PP)(f-\nabla\cdot(u\otimes u))=:G.
\end{equation}
The field $G$ is curl free. The radial potential
$p(x,t)=\int_0^1G(rx,t)\cdot x\dd r$ therefore recovers the pressure,
up to a function of time. We impose no $L^2$ condition on $p$ itself;
$\nabla p\in L^2$ excludes nonzero constant pressure gradients.

\begin{proposition}
\label{prop:local}\label{lem:Rlocal}
For each initial velocity in the stated class and each force smooth into
every $H^m$ on compact time intervals, equation~\eqref{eq:NS} has a unique
maximal smooth velocity, with pressure determined as above. If a solution
is defined on $[0,S)$, $S<\infty$, and
\begin{equation}\label{eq:criterion}
 \int_0^S\norm{u(t)}_{H^2(D)}^2\dd t<\infty,
\end{equation}
then it extends smoothly beyond $S$.
\end{proposition}
\begin{proof}
Apply Tao~\cite[Theorems~5.1(ii)--(iv), 5.4(ii)--(iv)]{tao2013} to
$\PP f$. The proof of Theorem~5.4(iv), p.~53, explicitly permits
$a\in H^k$ and $f\in C^j_tH^k_x$ for all $j,k$; the resulting smoothness
is on one common interval. Reduce viscosity to one by
$\tau=\nu t$, $\widetilde u=\nu^{-1}u$, and
$(\widetilde p,\widetilde f)=\nu^{-2}(p,f)$ with the corresponding time change.
For nonzero periodic mean, use Tao's mean reduction
\cite[Section~3, (36)--(38)]{tao2013}: set
$m(t)=\int_{\TT}a+\int_0^t\int_{\TT}f$ and translate by
$X(t)=\int_0^t m$, subtracting $m$ from the velocity and $m'$ from the force.
Pressure is recovered as above. Uniqueness and maximal development follow
from the same theorems and the overlap argument of Corollary~5.2.

For continuation, the standard $H^m$ energy estimate, Young's inequality
and \eqref{eq:Rproduct} give, for $m\ge3$,
\begin{equation}\label{eq:highcontinuation}
 (\norm{u}_{H^m})'\le C_{m,\nu}\norm{u}_{H^2}^2\norm{u}_{H^m}
                         +\norm{f}_{H^m}.
\end{equation}
Regularize the norm at zero. Gr\"onwall and \eqref{eq:criterion} bound
every $H^m$ norm uniformly up to $S$. Tao's quantitative local existence
then gives a uniform positive time when restarting at $t_0\uparrow S$;
the fixed force is smooth on $[0,S+1]$. Uniqueness glues one such solution
past $S$.
\end{proof}

\subsection{Energy and initial vanishing of the compact solution}
We use the fixed solution $(U,P,F)$ chosen after
Theorem~\ref{thm:packet}. Its energy estimate supplies the time-integrated
dissipation bound and the initial vanishing needed for rescaling.

\begin{lemma}\label{lem:packetenergy}
The building block satisfies
\[
 M:=\sup_{0\le t<1}\norm{U(t)}_2<\infty,\qquad
 D:=\norm{\nabla U}_{L^2((0,1)\times\R^3)}<\infty.
\]
Both $U$ and $P$ vanish on an initial time interval and therefore extend smoothly by zero to negative times.
\end{lemma}
\begin{proof}
The classical energy identity and $U(0)=0$ give, with
$N(t)=\int_0^t\|F(s)\|_2\dd s$, the bounds $\|U(t)\|_2\le N(t)$ and
\begin{equation}\label{eq:packetenergy}
 \norm{U(t)}_2^2+2\nu\int_0^t\norm{\nabla U(s)}_2^2\dd s
 \le 2\int_0^t\norm{F(s)}_2N(s)\dd s=N(t)^2.
\end{equation}
Compact support justifies the identity; regularizing $\|U\|_2$ justifies
the first bound at zero. Since $F$ is time compact, monotone convergence
gives finite dissipation up to time one. Before the support of $F$, the
same estimate gives $U=0$; the equation and compact support then give $P=0$.
Both fields therefore extend smoothly by zero to negative times.
\end{proof}

\section{The torus}\label{sec:torus}
All spatial norms in this section are on $\TT$, unless another domain is
shown. The spaces $\XX$, $\FF$, and $\BB_{\nu,a,T}$ are those defined in
Section~\ref{sec:preliminaries}.
\subsection{Statement of the density theorem}
\begin{theorem}[Sobolev density threshold]\label{thm:main}
Equip $\FF$ with its relative $L^1(0,\infty;H^s(\TT))$ norm topology.
\begin{enumerate}[label=(\roman*),nosep]
\item For every fixed $a\in\XX$, the set $\BB_{\nu,a,T}$ is dense if $s<1/2$.
\item For zero initial velocity,
\[
 \BB^0_{\nu,T}\text{ is dense in }\FF
 \quad\Longleftrightarrow\quad s<\tfrac12.
\]
\end{enumerate}
\end{theorem}

We first construct a blowup solution close to a given smooth solution. The
force estimates yield density below the stated threshold. A separate
small-force estimate then proves non-density at and above it.
\subsection{Localization and scaling}\label{sec:packet}
\begin{lemma}\label{lem:localization}
Let $B$ be a fixed coordinate ball whose closure lies inside a torus fundamental cube. If $z$ is smooth and supported in $B$, let $z_{\R}$ be its zero extension in that cube to $\R^3$, and let $z_{\TT}$ be its periodization. For $0<s<1$,
\begin{equation}\label{eq:localization}
 \norm{z_{\TT}}_{H^s(\TT)}
 \le C_{s,B}\bigl(\norm{z_{\R}}_2+
                   \norm{z_{\R}}_{\dot H^s(\R^3)}\bigr).
\end{equation}
The constant is uniform as the support shrinks inside $B$. At $s=0$ the $L^2$ norms are equal; at $s=1$ the corresponding equality holds for the $L^2$ gradient norms.
\end{lemma}
\begin{proof}
Sobolev norms on a compact manifold are equivalent to the norms obtained
by a fixed coordinate cover and partition of unity
\cite[Chapter~4, Sections~2--3]{taylor2011}. Choose one fixed cutoff equal
to one near $\overline B$ and supported in the coordinate chart. The resulting
local-to-global map is bounded from $H^s(\R^3)$ to $H^s(\TT)$.
Since $\|z\|_{H^s(\R^3)}\le C_s(\|z\|_2+\|z\|_{\dot H^s})$,
this gives \eqref{eq:localization}. The chart and cutoff are fixed, so the
constant is independent of a smaller support of $z$. At orders zero and
one, integration over the single copy of the support gives the stated identities.
\end{proof}

Choose a compact set $K_*$ containing $K$ and the spatial projection of
$\supp F$. Fix $x_0\in B$, where $B$ is the ball in
Lemma~\ref{lem:localization}. For sufficiently small $\eps>0$, require
\[
 2\eps^2<T,\qquad x_0+\eps K_*\subset B,
 \qquad t_\eps=T-\eps^2.
\]
We extend $F$ by zero to nonpositive source times; this extension is smooth
because its temporal support is compact in $(0,\infty)$. Using this
extension and the negative-time zero extensions of $U,P$ from
Lemma~\ref{lem:packetenergy}, define
\begin{align}
 U_\eps(x,t)&=\eps^{-1}U\left(\frac{x-x_0}{\eps},
                              \frac{t-t_\eps}{\eps^2}\right),\nonumber\\
 P_\eps(x,t)&=\eps^{-2}P\left(\frac{x-x_0}{\eps},
                              \frac{t-t_\eps}{\eps^2}\right),\label{eq:scaling}\\
 F_\eps(x,t)&=\eps^{-3}F\left(\frac{x-x_0}{\eps},
                              \frac{t-t_\eps}{\eps^2}\right).\nonumber
\end{align}
The first two fields are used for $t<T$; the force is defined for all positive times. Place the rescaled whole-space solution in $B$ and then periodize it. The torus contains a single copy of its support.

\begin{proposition}\label{prop:scaling}
The fields in \eqref{eq:scaling} solve the periodic momentum equation at viscosity $\nu$, start from zero, and have unbounded speed at $T$. A spatially constant pressure adjustment enforces the mean-zero normalization. Their norms satisfy
\begin{align}
 \norm{U_\eps}_{L^\infty(0,T;L^2)}
    &=\eps^{1/2}M,&
 \norm{\nabla U_\eps}_{L^2(0,T;L^2)}
    &=\eps^{1/2}D,\label{eq:packetEscale}\\
 \norm{F_\eps}_{L^q(0,\infty;L^p)}
    &=\eps^{\alpha(p,q)}\norm{F}_{L^q(0,\infty;L^p)},&
 \alpha(p,q)&=-3+\frac3p+\frac2q,\label{eq:packetFscale}
\end{align}
for $1\le p,q\le\infty$. For $0\le s\le1$,
\begin{equation}\label{eq:packetHs}
 \norm{F_\eps}_{L^1(0,\infty;H^s(\TT))}
 \le C_s\bigl(\eps^{1/2}+\eps^{1/2-s}\bigr).
\end{equation}
In particular the force tends to zero in $L^1_tH^s_x$ for every $s<1/2$, including negative $s$.
\end{proposition}
\begin{proof}
Writing $y=(x-x_0)/\eps$ and $\sigma=(t-t_\eps)/\eps^2$, every term of the momentum equation gains the common factor $\eps^{-3}$. Incompressibility is preserved and no viscosity rescaling occurs. The temporal extension is smooth by initial vanishing. The speed identity
$\norm{U_\eps(t)}_\infty=\eps^{-1}\norm{U(\sigma)}_\infty$
proves unboundedness as $t\uparrow T$.

At a fixed time the velocity has amplitude factor $\eps^{-1}$ and spatial volume factor $\eps^3$, giving the first identity in \eqref{eq:packetEscale}. Its gradient has amplitude factor $\eps^{-2}$. Squaring and integrating in both space and time gives
$\eps^{-4}\eps^3\eps^2=\eps$ times the original dissipation integral, proving the second identity.

For finite $p,q$, the force's spatial norm gains $\eps^{-3+3/p}$ and its time norm gains $\eps^{2/q}$. The positive-time domain includes the whole transformed support because $t_\eps>0$ and the original force vanishes at negative times. This proves \eqref{eq:packetFscale}; the essential-supremum cases follow by the same change of variables with zero reciprocal exponents.

Euclidean homogeneous Sobolev scaling gives, for $0<s<1$,
\[
 \norm{\eps^{-3}F((\,\cdot-x_0)/\eps,\sigma)}
          _{\dot H^s(\R^3)}
 =\eps^{-3/2-s}\norm{F(\cdot,\sigma)}_{\dot H^s(\R^3)}.
\]
The Fourier transform gives the squared homogeneous scaling factor
$\eps^{-6}\eps^6\eps^{-3-2s}=\eps^{-3-2s}$.
Time integration contributes $\eps^2$. The inhomogeneous $L^2$ term has order $\eps^{1/2}$. Lemma~\ref{lem:localization} therefore proves \eqref{eq:packetHs}. At $s=0$ use the mixed-norm identity, and at $s=1$ use the gradient scaling. For $s<0$, the Fourier definition gives $\norm{z}_{H^s}\le\norm{z}_2$, so the $s=0$ estimate proves convergence without any homogeneous negative-order claim.
\end{proof}

The force amplitude is $\eps^{-3}\norm{F}_\infty$, even though its integrated norm tends to zero below the critical order. The homogeneous $L^1_t\dot H^{1/2}_x$ norm is scale invariant. Section~\ref{sec:critical} supplies the regularity estimate needed at that order.

\subsection{Gluing to a given smooth solution}\label{sec:insertion}
Fix a given smooth solution $(v,\pi,g)$ smooth on $[0,T+\delta]$, where
$\delta>0$, $g\in\FF$, and $v(0)=a\in\XX$. We first make the given velocity vanish near the support of $U_\eps$. This prevents nonlinear interaction terms from introducing a singularity into the new force.

Choose $R>0$ with $K_*\subset B(0,R)$. The smooth Urysohn lemma
\cite[Corollary~1.3.4]{petersen} gives an open neighborhood $V$ of $K_*$
and smooth cutoff functions satisfying
\[
\begin{aligned}
 &\theta\in C_c^\infty(\R^3),\quad 0\le\theta\le1,\quad
   \theta|_V=1,\quad \supp\theta\subset B(0,R),\\
 &\eta\in C_c^\infty(\R),\quad 0\le\eta\le1,\quad
   \eta|_{[-1,1]}=1,\quad \supp\eta\subset(-2,2).
\end{aligned}
\]
These choices are fixed independently of $\eps$; their rescalings below
localize the modification near $x_0$ in space and near $T$ in time.

\begin{lemma}\label{lem:potential}
Let $v$ be smooth and divergence free in a spatial ball centered at $x_0$. In that ball, with $y=x-x_0$, define
\begin{equation}\label{eq:potential}
 A(x,t)=\int_0^1 r\,v(x_0+ry,t)\times y\dd r.
\end{equation}
Then $\nabla\times A=v$. With the fixed cutoffs $\theta$ and $\eta$ above, put
\begin{equation}\label{eq:cutoff}
 \theta_\eps(x)=\theta((x-x_0)/\eps),\quad
 \eta_\eps(t)=\eta((t-T)/\eps^2),\quad
 w_\eps=-\nabla\times(\eta_\eps\theta_\eps A).
\end{equation}
For sufficiently small $\eps$, $w_\eps$ is a smooth, divergence-free field supported inside the coordinate ball and in
$(T-2\eps^2,T+2\eps^2)$. During the active interval before blowup of $U_\eps$,
\begin{equation}\label{eq:bgzero}
 v+w_\eps=0
 \quad\hbox{on an open neighborhood of }\supp U_\eps(t).
\end{equation}
\end{lemma}
\begin{proof}
Differentiating the radial formula and using $\nabla\cdot v=0$ gives
$\nabla\times A=v$. Hence $w_\eps$ is divergence free, and on the common
plateau of the cutoffs it equals $-v$. For sufficiently small $\eps$ the
spatial support lies strictly inside the coordinate ball and the temporal
support lies in $(0,T+\delta)$, so extension by zero is smooth.
The rescaled building block is active only for $T-\eps^2\le t<T$,
where $\eta_\eps=1$; its support lies in the spatial plateau. This proves
\eqref{eq:bgzero} and the support assertions.
\end{proof}

\begin{lemma}\label{lem:correction}
For $w_\eps$ in Lemma~\ref{lem:potential}, define
\begin{align}
 H_\eps={}&\partial_tw_\eps-\nu\Delta w_\eps
 +(v\cdot\nabla)w_\eps+(w_\eps\cdot\nabla)v
 +(w_\eps\cdot\nabla)w_\eps.\label{eq:H}
\end{align}
The force $H_\eps$ is smooth across $T$ and extends by zero to all times outside its cutoff support. Its spatial support has volume $O(\eps^3)$, and its temporal support has length $O(\eps^2)$. For every spatial multi-index $\beta$ and integer $j\ge0$,
\begin{equation}\label{eq:derivativebounds}
 |\partial_t^j\partial_x^\beta w_\eps|
 \le C_{\beta,j}\eps^{-2j-|\beta|},\qquad
 |\partial_x^\beta H_\eps|
 \le C_\beta\eps^{-2-|\beta|}.
\end{equation}
In particular,
\begin{align}
 \norm{w_\eps}_{E_T}&\le C\eps^{3/2},\label{eq:wE}\\
 \norm{H_\eps}_{L^q(0,\infty;L^p)}
   &\le C_{p,q}\eps^{-2+3/p+2/q}
     =C_{p,q}\eps^{\alpha(p,q)+1},\label{eq:Hmixed}\\
 \norm{H_\eps}_{L^1(0,\infty;H^s)}
   &\le C_s\bigl(\eps^{3/2}+\eps^{3/2-s}\bigr)
       \quad(0\le s\le1).\label{eq:HHs}
\end{align}
All constants are independent of sufficiently small $\eps$.
\end{lemma}
\begin{proof}
Set $x=x_0+\eps z$, $t=T+\eps^2\sigma$. The radial potential has
$A(x,t)=\eps\mathcal A_\eps(z,\sigma)$, where
$\mathcal A_\eps$ and all its derivatives are uniformly bounded on the
fixed cutoff support. Thus
\[
 w_\eps(x,t)=W_\eps(z,\sigma),\qquad
 W_\eps=-\nabla_z\times(\eta(\sigma)\theta(z)\mathcal A_\eps).
\]
Substitution in \eqref{eq:H} gives $H_\eps(x,t)=\eps^{-2}Q_\eps(z,\sigma)$,
where $Q_\eps$ is uniformly smooth and supported in the same fixed cylinder:
its terms are $\partial_\sigma W_\eps$, $-\nu\Delta_zW_\eps$ and
transport terms with nonnegative powers of $\eps$ and bounded coefficients.
This proves \eqref{eq:derivativebounds}. Changing variables gives
\eqref{eq:wE}--\eqref{eq:Hmixed}; applying Sobolev scaling to $Q_\eps$
and then Lemma~\ref{lem:localization} gives \eqref{eq:HHs}.
The smooth cutoffs give the asserted zero extensions in time and space.
\end{proof}

\begin{theorem}\label{thm:insertion}
Let $(v,\pi,g)$ be a periodic solution smooth on $[0,T+\delta]$, with
$\delta>0$, $g\in\FF$ and $v(0)=a\in\XX$. Fix any nonempty coordinate ball.
For all sufficiently small $\eps>0$, there are $g_\eps\in\FF$ and a solution $u_\eps$ with
\begin{enumerate}[label=(\roman*),nosep]
\item $T_{\max}^{\nu}(a,g_\eps)=T$ and
$\limsup_{t\uparrow T}\norm{u_\eps(t)}_\infty=\infty$;
\item $u_\eps=v$ for $0\le t\le T-2\eps^2$;
\item $u_\eps-v$ divergence free and supported, for every $t<T$, in a ball of diameter $O(\eps)$ inside the chosen ball;
\item the simultaneous bounds
\end{enumerate}
\begin{align}
 \norm{u_\eps-v}_{E_T}
   &\le (M+D)\eps^{1/2}+C\eps^{3/2},\label{eq:Eclose}\\
 \norm{g_\eps-g}_{L^q_tL^p_x}
   &\le C_{p,q}\bigl(\eps^{\alpha(p,q)}
                    +\eps^{\alpha(p,q)+1}\bigr),\label{eq:Fclose}\\
 \norm{g_\eps-g}_{L^1_tH^s_x}
   &\le C_s\bigl(\eps^{1/2-s}+\eps^{3/2-s}\bigr)
                         \quad(0\le s<1/2).\label{eq:Hsclose}
\end{align}
Here force norms are taken over $(0,\infty)$, and
$\alpha(p,q)=-3+3/p+2/q$ for $1\le p,q\le\infty$.
For $s<0$, the force difference also tends to zero in $L^1_tH^s_x$.
\end{theorem}
\begin{proof}
Choose the rescaled blowup solution and the correction from Lemmas~\ref{lem:potential}--\ref{lem:correction}, with all supports inside the prescribed ball, and set
\begin{equation}\label{eq:insertion}
 u_\eps=v+w_\eps+U_\eps,\qquad
 p_\eps=\pi+P_\eps,\qquad
 g_\eps=g+H_\eps+F_\eps.
\end{equation}
Initially use the compact representative $P_\eps$; subtract its spatial mean to normalize $p_\eps$.

Let $b_\eps=v+w_\eps$. Expanding the equation for the given smooth solution and \eqref{eq:H} gives
\[
 \partial_tb_\eps+(b_\eps\cdot\nabla)b_\eps
 -\nu\Delta b_\eps+\nabla\pi=g+H_\eps.
\]
Adding the building block equation produces the equation for $u_\eps$ except for
\[
 (b_\eps\cdot\nabla)U_\eps+(U_\eps\cdot\nabla)b_\eps.
\]
Both terms are identically zero. On a neighborhood of $\supp U_\eps$ this follows from $b_\eps=0$, including its derivatives, by \eqref{eq:bgzero}. Outside that support $U_\eps$ and its derivatives vanish, including at the support boundary by smoothness. Before its starting time it vanishes everywhere. Thus \eqref{eq:NS} holds exactly for \eqref{eq:insertion} on $[0,T)$.

Each summand of the velocity is divergence free. The correction is zero for
$t\le T-2\eps^2$, and the building block is zero before $T-\eps^2$, so the initial datum and the stated earlier history are unchanged. On the active building block support, $u_\eps=U_\eps$. Hence
\[
 \norm{u_\eps(t)}_\infty\ge\norm{U_\eps(t)}_\infty
\]
and the maximum norm is unbounded as $t\uparrow T$.

The velocity is smooth on every closed interval before blowup. The force $F_\eps$ is globally smooth by its original definition, and $H_\eps$ is globally smooth by Lemma~\ref{lem:correction}. Their time supports remain compact subsets of $(0,\infty)$ when $\eps$ is sufficiently small. Therefore $g_\eps\in\FF$, independently of the absence of a classical extension of $u_\eps$ at $T$. Uniqueness in Proposition~\ref{prop:local} identifies the constructed velocity with the maximal solution for $(a,g_\eps)$ up to every $T'<T$. Its lifespan is at least $T$, and the unbounded maximum norm excludes extension beyond $T$. Thus the lifespan is exactly $T$.

Finally, the triangle inequality, Proposition~\ref{prop:scaling} and
Lemma~\ref{lem:correction} give \eqref{eq:Eclose}--\eqref{eq:Hsclose}. In
\eqref{eq:Hsclose}, the lower-order inhomogeneous terms are absorbed because
$0<\eps\le1$ and $s\ge0$. For negative $s$, use
$\norm{z}_{H^s}\le\norm{z}_2$ and the $s=0$ estimate.
\end{proof}

\subsection{Density below the critical order}\label{sec:density}
\begin{proposition}\label{prop:density}
For $a\in\XX$, $\nu>0$, $T>0$ and $s<1/2$,
\[
 \forall g\in\FF\ \forall\rho>0\ \exists f\in\FF:
 \quad \norm{f-g}_{L^1_tH^s_x}<\rho,\qquad
       T_{\max}^{\nu}(a,f)\le T.
\]
\end{proposition}
\begin{proof}
Fix $g$ and $\rho$. There are two exhaustive cases.

If $T_{\max}^{\nu}(a,g)\le T$, choose $f=g$. It already belongs to
$\BB_{\nu,a,T}$ and the norm difference is zero.

If $T_{\max}^{\nu}(a,g)>T$, choose $\delta>0$ so that
$T+\delta<T_{\max}^{\nu}(a,g)$, and use its smooth solution as the given smooth solution in Theorem~\ref{thm:insertion}. For every sufficiently small $\eps$, that theorem supplies $g_\eps\in\BB_{\nu,a,T}$ with lifespan exactly $T$.
For $0\le s<1/2$, \eqref{eq:Hsclose} tends to zero. For $s<0$, its $s=0$ case and the embedding $L^2\hookrightarrow H^s$ do so. Choose $\eps$ so that the force difference is less than $\rho$ and set $f=g_\eps$.

Both choices preserve $a$, $\nu$, and $T$, and give the required approximation.
\end{proof}

\subsection{Critical regularity and completion of the proof}\label{sec:critical}
Non-density at and above the threshold is a consequence of classical
stability of the zero solution in the critical topology. For a direct
periodic robustness theorem with $L^2_tH^{-1/2}_x$ force perturbations, see
Mar\'in-Rubio, Robinson and Sadowski~\cite[Theorem~3, author manuscript
pp.~6--7]{mrrs}. The $L^1_tH^{1/2}_x$ small-force result needed here follows
from the forced Fujita--Kato theory in Danchin~\cite[Theorem~2.3.1,
pp.~47--49]{danchin}; we record only the adaptation of its hypotheses.

\begin{proposition}\label{prop:critical}
There is $c>0$, depending only on the unit torus and the stated norm conventions, such that
\begin{equation}\label{eq:smallcritical}
 g\in\FF,\qquad
 \rho:=\norm{g}_{L^1(0,\infty;H^{1/2})}<c\nu
\end{equation}
imply $T_{\max}^{\nu}(0,g)=\infty$.
\end{proposition}
\begin{proof}
First remove the mean as in Tao~\cite[Section~3, (36)--(38)]{tao2013}.
Set $m(t)=\int_0^t\overline g(\tau)\dd\tau$ and
$b(t)=\int_0^tm(\tau)\dd\tau$. In the translated coordinates,
$v(t,x)=u(t,x+b(t))-m(t)$ solves the mean-zero periodic equation with force
$h(t,x)=g(t,x+b(t))-\overline g(t)$. Translation is an isometry and removal
of the zero mode decreases the norm, so
$\|h\|_{L^1_tH^{1/2}_x}\le\rho$ and $|m(t)|\le\rho$.

Apply the critical small-data fixed-point theorem
\cite[Theorem~2.3.1 and its proof]{danchin} with $p=r=2$.
Here $\dot B^{1/2}_{2,2}=\dot H^{1/2}$, and Minkowski's inequality gives
$\|h\|_{\widetilde L^1\dot H^{1/2}}\le\|h\|_{L^1\dot H^{1/2}}$.
The same proof applies to mean-zero periodic fields using the periodic heat
estimates in \cite[Section~2.2.4, p.~44]{danchin}; the zero mode has already
been removed. It gives a global critical solution when $\rho<c\nu$.
Smooth data retain their higher regularity by the local theory cited in
Proposition~\ref{prop:local}. Restoring the translation and bounded mean
therefore gives the required global classical solution.
\end{proof}

\begin{corollary}\label{cor:nondensity}
For every $s\ge1/2$ and $T>0$, $\BB^0_{\nu,T}$ is not dense in
$\FF$ with the relative $L^1_tH^s_x$ topology.
\end{corollary}
\begin{proof}
At $s=1/2$, the set
\[
 \{g\in\FF:\norm{g}_{L^1_tH^{1/2}_x}<c\nu\}
\]
is a nonempty relative open ball, containing zero, and is disjoint from
$\BB^0_{\nu,T}$ by Proposition~\ref{prop:critical}.
For $s\ge1/2$, the Fourier weights imply
$\norm{g(t)}_{H^{1/2}}\le\norm{g(t)}_{H^s}$.
The ball $\norm{g}_{L^1_tH^s_x}<c\nu$ is therefore also disjoint from the set of forces producing breakdown. A dense subset cannot miss a nonempty open set.
\end{proof}

\begin{proof}[Proof of Theorem~\ref{thm:main}]
Proposition~\ref{prop:density} gives density for every fixed smooth initial velocity when $s<1/2$. Corollary~\ref{cor:nondensity} gives non-density for zero initial velocity when $s\ge1/2$. These are the two assertions of the theorem.
\end{proof}
The converse has been proved for zero initial velocity. The behavior for general nonzero initial velocities at or above the critical order remains outside this classification.

\subsection{Other force norms, trajectory approximation, and data pairs}
\begin{corollary}\label{cor:mixed}
For every fixed $a\in\XX$, the set $\BB_{\nu,a,T}$ is dense in $\FF$ with its relative $L^q(0,\infty;L^p)$ topology whenever
\begin{equation}\label{eq:mixedregion}
 1\le p,q\le\infty,\qquad \frac3p+\frac2q>3.
\end{equation}
\end{corollary}
\begin{proof}
Use the same two cases as in Proposition~\ref{prop:density}. In the case of a solution smooth through the prescribed time, \eqref{eq:Fclose} tends to zero because
$\alpha(p,q)>0$ and hence $\alpha(p,q)+1>0$.
\end{proof}
The sufficient region includes $L^1_tL^2_x$ and $L^2_tL^{4/3}_x$. The other mixed Lebesgue topologies require further analysis.

\begin{corollary}\label{cor:closure}
Fix $a\in\XX$. Let $\mathcal R_{a,T}$ be the trajectories with initial velocity $a$ and forces in $\FF$ that are smooth through $T$.
Let $\mathcal S_{a,T}$ be the trajectories with the same initial velocity and forces in $\FF$ that are smooth on $[0,T)$, have finite $E_T$ norm, and have unbounded maximum velocity at $T$. Then
\begin{equation}\label{eq:closure}
 \mathcal R_{a,T}\subseteq\overline{\mathcal S_{a,T}}^{\,E_T}.
\end{equation}
For each given pair $(v,g)$ and every fixed $s<1/2$, the approximating pairs can be chosen with
\[
 (u_\eps,g_\eps)\longrightarrow(v,g)
 \quad\hbox{in }E_T\times L^1(0,\infty;H^s).
\]
\end{corollary}
\begin{proof}
A given solution smooth through $T$ is smooth through $T+\delta$ for some
$\delta>0$, by local continuation; equivalently one may take this as the meaning of smoothness through the endpoint. Apply Theorem~\ref{thm:insertion}. Its energy estimate gives convergence in $E_T$, and its force estimates give simultaneous convergence. The given smooth solution has finite $E_T$ norm and the building block and correction do also, so $u_\eps$ belongs to the specified ambient energy space.

To verify the claimed interpretation of the norm, note that
\[
 \norm{z}_{L^2(0,T;L^2)}
 \le T^{1/2}\norm{z}_{L^\infty(0,T;L^2)}.
\]
Thus $E_T$ controls both terms of $L^2_tH^1_x$, while the usual sum norm controls $E_T$ directly. This proves the equivalence used in \eqref{eq:closure}. No endpoint value or classical continuation of $u_\eps$ is implied by convergence in this norm.
\end{proof}

\begin{proposition}\label{prop:projection}
Define
\[
 \mathfrak B_{\nu,T}
 =\{(a,f)\in\XX\times\FF:T_{\max}^{\nu}(a,f)\le T\}.
\]
If $s<1/2$, this set is dense in the product of any topology on $\XX$ and the relative $L^1_tH^s_x$ topology on $\FF$. Its projection onto $\XX$ is all of $\XX$.
The family obtained while requiring $a=0$ projects instead to the singleton $\{0\}$.
\end{proposition}
\begin{proof}
For every fixed $a$, Proposition~\ref{prop:density} makes the force fiber
dense and nonempty. Every product neighborhood of $(a,g)$ therefore meets
$\mathfrak B_{\nu,T}$ with its first coordinate still equal to $a$.
This proves both density and surjectivity of the projection. Restricting
to $a=0$ gives the last assertion.
\end{proof}

The projection statement has the quantifier order
$\forall a\,\exists f$, not $\exists f\,\forall a$.
For a fixed prescribed force $f_*$, the set
\[
 \{a\in\XX:T_{\max}^{\nu}(a,f_*)\le T\}
\]
is a different object and is not classified by these force-changing constructions.
There is also a distinction between the two terminal-time assertions.
Theorem~\ref{thm:insertion} gives singularity exactly at $T$ around every given solution smooth through $T$. Proposition~\ref{prop:density} uses breakdown by $T$, because a given solution may already break down earlier and is then left unchanged. It does not prove that forces producing velocity blowup exactly at $T$
are dense near a force whose solution breaks down earlier.
\begin{remark}\label{rem:peaks}
The force convergence in Theorem~\ref{thm:insertion} is compatible with diverging pointwise amplitudes. Since the original force is nonzero,
\[
 \norm{g_\eps-g}_{L^\infty_{t,x}}
 \ge\eps^{-3}\norm{F}_\infty-C\eps^{-2}\longrightarrow\infty.
\]
Here the first term is the amplitude of $F_\eps$ and the second bounds $H_\eps$. The force $F$ is nonzero by~\eqref{eq:packetenergy}, since $U$ blows up. Hence the approximating smooth forces have unbounded amplitudes and cannot obey uniform bounds on all derivatives.
\end{remark}

\subsection{Gluing in a bounded domain}
For the following bounded-domain statement, $H^s(\Omega)$ is the space
of restrictions of $H^s(\R^3)$ distributions, with quotient norm
\begin{equation}\label{eq:restriction-norm}
 \norm{z}_{H^s(\Omega)}=\inf_{Z|_\Omega=z}\norm{Z}_{H^s(\R^3)}.
\end{equation}
This convention applies also to negative orders. At order zero it is the
usual $L^2(\Omega)$ norm.
For every fixed compact $K\Subset\Omega$ and $s\in\R$, smooth fields supported
in $K$ obey
\begin{equation}\label{eq:zero-extension}
 \norm{z}_{H^s(\Omega)}\le\norm{E_0z}_{H^s(\R^3)}
 \le C_{s,K,\Omega}\norm{z}_{H^s(\Omega)},
 \qquad \operatorname{supp}z\subset K,
\end{equation}
Here $E_0z$ denotes extension by zero. Choose $\chi\in C_c^\infty(\Omega)$ equal to one near $K$ by the
smooth Urysohn lemma~\cite{petersen}. Multiplication by a fixed smooth cutoff
is bounded on $H^s(\R^3)$ for every real $s$
\cite[Chapter~4, Section~2, (2.19)]{taylor2011}. For every extension $v$ of $z$,
$\chi v=E_0z$; taking the infimum over extensions proves the second
inequality in \eqref{eq:zero-extension}. The first follows from the quotient norm.
The constant depends on the fixed cutoff and $s$, not on a smaller support
or a shrinking scale $\eps$. Applying the pointwise bounds in time gives the
same comparison for all the mixed norms in \eqref{eq:time-norms}, provided
$K$ is fixed for all times. No bounded zero-extension operator on arbitrary
$H^s(\Omega)$ is asserted. All vector and tensor norms in these definitions
are the square root of the sum of the squared component norms.

The next corollary extends the local gluing construction to a bounded
container. Placing the modification strictly inside an interior ball leaves
the velocity unchanged near the boundary, so the no-slip condition is
preserved. The starting point is a given compatible smooth solution; the
construction changes its interior evolution while retaining its boundary values.

\begin{corollary}\label{cor:boundary}
Let $\Omega\subset\R^3$ be a bounded box or a bounded smooth domain.
Suppose that, for some $\delta>0$, $(v,\pi,g)$ is a smooth no-slip
solution on $[0,T+\delta]$, and that $g$ is smooth on
$\overline\Omega\times[0,\infty)$ with temporal support compact in
$(0,\infty)$.
Here smoothness on a closed spacetime slab means restriction of a
$C^\infty$ field from an open neighborhood of that slab; this convention
also specifies smoothness at the edges and corners of a box. We assume
$v,\pi$ have this regularity on
$\overline\Omega\times[0,T+\delta]$, and $g$ on each finite closed slab.
The equation and incompressibility hold in $\Omega$, and
$v|_{\partial\Omega}=0$. Pressure may be normalized by zero spatial mean.
Existence of this compatible smooth solution is an assumption of the corollary.
The construction in Theorem~\ref{thm:insertion} applies inside any prescribed interior ball, preserving the initial velocity and no-slip boundary values.
The energy estimates use $L^2(\Omega)$ and the force estimates use
$L^1(0,\infty;H^s(\Omega))$, with the restriction norm
\eqref{eq:restriction-norm}. For every real $s$, the domain and zero-extended
force-difference norms are comparable by \eqref{eq:zero-extension}, with a
constant independent of $\eps$; the convergence assertion remains $s<1/2$.
\end{corollary}
\begin{proof}
The vector potential and all cutoff estimates are local. Choose the correction and building block supports in a smaller closed ball strictly inside $\Omega$.
The computation of the momentum equation in Theorem~\ref{thm:insertion} is unchanged. The new velocity equals the given smooth solution in a fixed boundary collar for all times before blowup, so the no-slip boundary values are preserved. The force correction is supported in the same interior region and has the same globally smooth time extension.

For $0\le s<1/2$, the Euclidean scaling proof directly bounds the zero
extensions of the force differences in $L^1(0,\infty;H^s(\R^3))$.
For $s<0$, use $\norm{E_0(g_\eps-g)}_{H^s(\R^3)}
\le\norm{E_0(g_\eps-g)}_{L^2(\R^3)}$ at each time.
All these force-difference supports, including those after $T$, lie in one
fixed compact interior ball $K$ for every sufficiently small $\eps$.
Consequently \eqref{eq:zero-extension} gives the asserted domain estimates
and the scale-independent comparison after time integration. The energy estimates follow by integrating over that ball. Finally, the classical difference-energy estimate and Gr\"onwall give uniqueness on the no-slip domain; the boundary terms vanish. Thus the constructed solution is singular exactly at $T$.
\end{proof}

\subsection{Further constructions}\label{sec:variants}
\begin{proposition}\label{prop:affine}
Fix the building block $(U,P,F)$ of Theorem~\ref{thm:packet}, and fix a spacetime cylinder $Q=B_0\times(\tau_0,\tau_1)$ with
$0<\tau_0<\tau_1<1$ and $B_0\Subset\R^3$ an open ball.
For any $b\in C_c^\infty(Q;\R^3)$ with $\nabla\cdot b=0$, define
\begin{align}
 \widetilde U&=U+b,\qquad \widetilde P=P,\nonumber\\
 \widetilde F&=F+\partial_tb-\nu\Delta b
       +(U\cdot\nabla)b+(b\cdot\nabla)U+(b\cdot\nabla)b.\label{eq:affine}
\end{align}
These fields give an infinite-dimensional affine family of distinct velocities with zero initial data, finite energy and dissipation, and the same late singular behavior as $U$. For each fixed nonzero $b$, the family obtained by replacing $b$ with $\lambda b$ is non-isolated as $\lambda\to0$ in every fixed smooth seminorm of the velocity and force differences on their compact support.
\end{proposition}
\begin{proof}
Expanding the momentum equation for $U+b$ gives exactly the force in
\eqref{eq:affine}; its divergence remains zero. Every correction term is supported in $\supp b$, a compact subset of $Q$. The original blowup solution is smooth with bounded derivatives of every fixed order on a neighborhood of that set. Thus the force correction extends smoothly by zero to the whole positive time axis, including time one, despite the singularity of $U$ there.

Since $b=0$ near time zero and near time one,
$\widetilde U(0)=0$ and $\widetilde U=U$ on a fixed late interval.
The same unbounded-speed limit holds. The energy and dissipation norms of $b$ are finite by compact smooth support, so the triangle inequality gives the corresponding bounds for $\widetilde U$.

To verify infinite dimensionality explicitly, choose countably many disjoint small balls inside $B_0$. In each ball choose a smooth compactly supported vector potential with nonzero curl, and multiply that curl by a fixed nonzero time bump in $(\tau_0,\tau_1)$. These divergence-free fields are linearly independent because their spatial supports are disjoint. Their finite linear combinations belong to the allowed class. The map $b\mapsto U+b$ is injective and affine, so its image is infinite dimensional.

For the non-isolation statement, the velocity difference is $\lambda b$, while the force difference has the form
\[
 \lambda L_U b+\lambda^2(b\cdot\nabla)b,\qquad
 L_Ub=\partial_tb-\nu\Delta b+(U\cdot\nabla)b+(b\cdot\nabla)U.
\]
On the fixed compact support, all coefficients and their derivatives are bounded. Hence for each integer $m\ge0$ the $C^m$ norm of the force difference is at most $C_m|\lambda|+C_m'\lambda^2$, and the velocity difference is at most $C_m''|\lambda|$. Both tend to zero. The forces vary in this family; no open singularity basin for one fixed force follows.
\end{proof}

\begin{proposition}\label{prop:multiple}
Fix finitely many disjoint interior balls $B_1,\ldots,B_N$ in the torus or in a bounded three-dimensional domain. At fixed $\nu>0$ and $T>0$, there is a smooth forced solution with zero initial velocity, smooth on $[0,T)$, with finite energy and dissipation, such that
\[
 \limsup_{t\uparrow T}\norm{u(t)}_{L^\infty(B_j)}=\infty
 \qquad(1\le j\le N).
\]
In a bounded domain its boundary condition is homogeneous no-slip.
\end{proposition}
\begin{proof}
Choose $x_j\in B_j$ and scales $\eps_j>0$ sufficiently small that the complete scaled compact set $x_j+\eps_jK_*$ lies strictly inside $B_j$ and $\eps_j^2<T$. Apply \eqref{eq:scaling} to each building block with start time $T-\eps_j^2$, and denote the resulting fields by $U_j,P_j,F_j$.
Set
\[
 u=\sum_{j=1}^N U_j,\qquad p=\sum_{j=1}^N P_j,\qquad f=\sum_{j=1}^N F_j.
\]
For $i\ne j$, the supports of $U_i$ and $U_j$ are separated, so
$(U_i\cdot\nabla)U_j=0$. The sum therefore satisfies the momentum equation exactly. All force supports are compact in positive time and all forces are smooth. Initial vanishing gives $u(0)=0$.

On $B_j$ the velocity equals $U_j$, proving the stated separate limsup for each ball. Disjoint supports also give
\[
 \sup_{t<T}\norm{u(t)}_2^2\le M^2\sum_{j=1}^N\eps_j,\qquad
 \int_0^T\norm{\nabla u(t)}_2^2\dd t=D^2\sum_{j=1}^N\eps_j.
\]
These sums are finite. In a bounded domain, all fields vanish in a boundary collar; hence $u=0$ on the boundary. On the torus a spatially constant pressure adjustment gives the chosen gauge.
\end{proof}
Only finitely many solutions are superposed, all with terminal time $T$. The sequences along which their maxima diverge may differ; the conclusion is a separate limsup statement in each ball.

\begin{proposition}\label{prop:conservative}
On the torus, let $f=-\nabla\phi$ for a globally defined periodic scalar
$\phi$. On a bounded domain, let $f=-\nabla\phi$ and impose homogeneous no-slip velocity. In either case, any smooth solution with zero initial velocity is identically zero on its classical lifespan.
\end{proposition}
\begin{proof}
Incompressibility and periodicity or the zero normal boundary trace give
\[
 \int f\cdot u
 =-\int\nabla\phi\cdot u=0.
\]
The classical energy identity, integrated from zero, becomes
\[
 \frac12\norm{u(t)}_2^2
 +\nu\int_0^t\norm{\nabla u(s)}_2^2\dd s=0.
\]
Both terms are nonnegative, so $u(t)=0$ at every time before blowup.
The periodic-potential condition excludes affine potentials, whose gradients need not be removable by the allowed periodic pressure gauge.
\end{proof}

Thus a nontrivial solution with zero initial velocity requires a nonconservative force.

\section{The whole space}\label{sec:whole}
We now work on $\R^3$, with the data classes $\mathcal X_{\R}$ and
$\mathcal F_{\R}$ from Section~\ref{sec:preliminaries}. The construction
is again spatially local. The new issues are the behavior of negative
Sobolev norms near frequency zero and the absence of a spectral gap.
\subsection{Statement of the density theorem}
\begin{theorem}[Two whole-space Sobolev thresholds]\label{thm:Rmain}
Fix $\nu,T>0$, let $q\in\{1,2\}$, and set $s_q=2/q-3/2$. In the relative $L^q(0,\infty;H^s(\R^3))$ topology on $\mathcal F_{\R}$:
\begin{enumerate}[label=(\roman*),nosep]
\item For every fixed $a\in\mathcal X_{\R}$, $\mathcal B^\R_{\nu,a,T}$ is dense if $s<s_q$.
\item For zero initial velocity, $\mathcal B^\R_{\nu,0,T}$ is dense if and only if $s<s_q$.
\end{enumerate}
Thus the thresholds are $1/2$ for $L^1_tH^s_x$ and $-1/2$ for $L^2_tH^s_x$. Around every given solution smooth through $T$, the approximating solution can have the same initial velocity, the same earlier history, singularity exactly at $T$, and a velocity difference tending to zero in $E_T$.
\end{theorem}

The force classes remain smooth, equipped with the relative norm topology. For $q=2$, the converse uses a regular neighborhood whose radius depends on $T$.

\subsection{Local construction and negative Sobolev estimates}
The compact solution of Theorem~\ref{thm:packet} is already defined on
$\R^3$. Use the scaling~\eqref{eq:scaling} directly, without
periodization. The energy identities~\eqref{eq:packetEscale} and mixed
force identity~\eqref{eq:packetFscale} follow from the same changes of
variables. The vector potential in Lemma~\ref{lem:potential} is local,
so its proof applies in any Euclidean ball. In
Lemma~\ref{lem:correction}, the support, derivative, energy, and mixed
Lebesgue bounds are also local; only the final transfer of fractional
norms to the torus is unnecessary. Thus the same $w_\eps$ and $H_\eps$
are available on $\R^3$, with the same uniform bounds. We record the
result and prove the additional negative-order estimates.
\Needspace{16\baselineskip}
\begin{theorem}\label{thm:Rinsert}
Let $(v,\pi,g)$ be the solution for $a\in\mathcal X_{\R}$ and $g\in\mathcal F_{\R}$, smooth through $T+\delta$ for some $\delta>0$. Fix any nonempty open ball $B\subset\R^3$. For all sufficiently small $\eps>0$, there are $g_\eps\in\mathcal F_{\R}$ and a solution $u_\eps$ such that
\[
 T^\nu_{\max,\R}(a,g_\eps)=T,\quad
 \limsup_{t\uparrow T}\norm{u_\eps(t)}_\infty=\infty,\quad
 u_\eps=v\ \ (0\le t\le T-2\eps^2).
\]
The velocity difference is supported inside $B$ at every $t<T$, and $g_\eps-g\in C_c^\infty(B\times(0,\infty))$. Moreover,
\begin{equation}\label{eq:REclose}
 \norm{u_\eps-v}_{E_T}\le(M+D)\eps^{1/2}+C\eps^{3/2}.
\end{equation}
For $q\in\{1,2\}$ the force difference tends to zero in $L^q_tH^s_x$ whenever $s<2/q-3/2$. All of these conclusions hold for the same family of glued solutions.
\end{theorem}
\begin{proof}
Place the fields in \eqref{eq:scaling} inside $B$, and choose the compact cutoffs of Lemma~\ref{lem:potential}. Define $A,w_\eps,H_\eps$ by \eqref{eq:potential}, \eqref{eq:cutoff} and \eqref{eq:H}. These formulas only use $v$ on a fixed compact neighborhood of $B$ and of time $T$, so the preceding lemmas apply. In particular $v+w_\eps=0$ on a neighborhood of the active building block support. Put
\[
 u_\eps=v+w_\eps+U_\eps,\qquad
 p_\eps=\pi+P_\eps,\qquad
 g_\eps=g+H_\eps+F_\eps.
\]
Writing $b_\eps=v+w_\eps$, expansion of the equation leaves only the cross-advection terms $(b_\eps\cdot\nabla)U_\eps+(U_\eps\cdot\nabla)b_\eps$. They vanish on a neighborhood of the building block support because $b_\eps=0$ there. Outside that support, the building block and its derivatives vanish; smoothness gives the same conclusion at the support boundary. When the building block is inactive it is identically zero. Thus the two terms vanish pointwise everywhere, and the equation is exact. The pressure difference may be chosen to be the compact scalar $P_\eps$: the equation gives an $L^2$ pressure gradient, and applying $I-\PP$ recovers \eqref{eq:Rpressure}. Both force corrections are globally smooth and spacetime compact, including across $T$, so their sum preserves membership in $\mathcal F_{\R}$.

The initial value and earlier history are unchanged. At each $t<T$ the compact blowup solution and cutoff correction lie in $H^\infty$, and the maximum norm of $U_\eps$ is unbounded at $T$. Proposition~\ref{prop:local} identifies the solution with the unique maximal solution. An extension through $T$ would be bounded in $C_tH^2$ on a neighborhood of $T$, hence bounded in $L^\infty_x$ by \eqref{eq:Rproduct}, contradicting the blowup of $U_\eps$. Thus its maximal lifespan is exactly $T$. Equation~\eqref{eq:packetEscale} and the correction estimate in Lemma~\ref{lem:correction} give \eqref{eq:REclose} by the triangle inequality.

It remains to justify every asserted force topology, including negative orders. Let
\[
 \beta(q,s)=\frac2q-\frac32-s.
\]
For $0\le s\le1$, Euclidean scaling and
$\norm{z}_{H^s}\le C_s(\norm{z}_2+\norm{z}_{\dot H^s})$ give
\begin{align}
 \norm{F_\eps}_{L^q_tH^s_x}
 &\le C_{q,s}\bigl(\eps^{2/q-3/2}+\eps^{\beta(q,s)}\bigr),\nonumber\\
 \norm{H_\eps}_{L^q_tH^s_x}
 &\le C_{q,s}\bigl(\eps^{2/q-1/2}+\eps^{\beta(q,s)+1}\bigr).
 \label{eq:RpositiveScale}
\end{align}
The second line uses the uniformly smooth compact rescaled profile of Lemma~\ref{lem:correction}, whose amplitude is $\eps^{-2}$.

For $-3/2<s<0$, every smooth compact profile has finite $\dot H^s$ norm. Indeed, on $|\xi|<1$ its Fourier transform is bounded by its $L^1$ norm and $\int_{|\xi|<1}|\xi|^{2s}\dd\xi<\infty$; on $|\xi|\ge1$ its $L^2$ norm suffices. These bounds are uniform for the rescaled correction profiles, which have common compact support and uniform derivatives. Since
$(1+|\xi|^2)^s\le|\xi|^{2s}$, scaling gives
\begin{equation}\label{eq:RnegativeScale}
 \norm{F_\eps}_{L^q_tH^s_x}\le C_{q,s}\eps^{\beta(q,s)},\qquad
 \norm{H_\eps}_{L^q_tH^s_x}\le C_{q,s}\eps^{\beta(q,s)+1}.
\end{equation}
The time norms run over the entire positive axis. Each rescaled force has time support of length $O(\eps^2)$, including any part after $T$.

For $q=1$, \eqref{eq:RpositiveScale} proves convergence for $0\le s<1/2$; all $s<0$ follow from $\norm{z}_{H^s}\le\norm{z}_2$. For $q=2$, \eqref{eq:RnegativeScale} proves convergence for $-3/2<s<-1/2$. If $s\le-3/2$, choose $r\in(-3/2,-1/2)$ with $r>s$ and use $\norm{z}_{H^s}\le\norm{z}_{H^r}$. This last step avoids making any false homogeneous scaling assertion at indices where a generic compact profile can have an infinite homogeneous norm.
\end{proof}

\subsection{The critical estimate for time-integrable forcing}
\begin{proposition}\label{prop:Rcritical1}
There is a universal $c>0$ such that, for $a\in\mathcal X_{\R}$ and $f\in\mathcal F_{\R}$,
\[
 \norm{a}_{\dot H^{1/2}}+\norm{f}_{L^1_t\dot H^{1/2}_x}<c\nu
 \quad\Longrightarrow\quad T^\nu_{\max,\R}(a,f)=\infty.
\]
In particular, for $a=0$, the ball $\norm{f}_{L^1_tH^{1/2}_x}<c\nu$ consists of globally regular inputs.
\end{proposition}
\begin{proof}
This is the $p=r=2$ case of Danchin's forced critical small-data theorem
\cite[Theorem~2.3.1, pp.~47--49]{danchin}. Indeed,
$\dot B^{1/2}_{2,2}=\dot H^{1/2}$ and
$\|f\|_{\widetilde L^1\dot H^{1/2}}\le\|f\|_{L^1\dot H^{1/2}}$;
the Leray projection is bounded on these spaces. The theorem gives a global
critical solution, and persistence of regularity for the smooth data in our
classes identifies it with the maximal classical solution of
Proposition~\ref{prop:local}. Finally
$\|f\|_{\dot H^{1/2}}\le\|f\|_{H^{1/2}}$ gives the stated open ball.
\end{proof}

\subsection{The critical estimate for square-integrable forcing}
\begin{proposition}\label{prop:Rcritical2}
For each $\nu,S>0$ there is $r_{\nu,S}>0$ such that
\[
 f\in\mathcal F_{\R},\qquad
 \norm{f}_{L^2(0,\infty;H^{-1/2})}<r_{\nu,S}
 \quad\Longrightarrow\quad T^\nu_{\max,\R}(0,f)>S.
\]
One can take $r_{\nu,S}=c\nu^{3/2}e^{-C\nu S}$ for suitable universal positive constants $c,C$.
\end{proposition}
\begin{proof}
Use the inhomogeneous heat estimate
\cite[Theorem~2.2.5, p.~44]{danchin} and the critical bilinear fixed-point
argument \cite[Lemma~2.3.2 and the proof of Theorem~2.3.1, pp.~48--49]{danchin}.
We spell out the low-frequency and viscosity adjustments. Normalize viscosity
by $u(t,x)=\nu v(\nu t,x)$ and $F(\tau,x)=\nu^{-2}f(\tau/\nu,x)$.
On $0\le\tau\le L:=\nu S$, set
\[
 X_L=C([0,L];H^{1/2})\cap L^2(0,L;H^{3/2}),
 \qquad
 \mathcal H F(t)=\int_0^t e^{(t-r)\Delta}\mathbb P F(r)\dd r.
\]
The cited heat estimate and the Sobolev product law give, with the sum norm
on $X_L$,
\[
 \|\mathcal H F\|_{X_L}\le Ce^{CL}\|F\|_{L^2H^{-1/2}},\qquad
 \|\mathcal B(v,w)\|_{X_L}\le Ce^{CL}\|v\|_{X_L}\|w\|_{X_L},
\]
where $\mathcal B(v,w)=-\mathcal H\operatorname{div}(v\otimes w)$.
For the second bound use $X_L\hookrightarrow L^4_tH^1_x$ and
$H^1\cdot H^1\hookrightarrow H^{1/2}$; the exponential is a convenient
upper bound for the finite-time constants in the inhomogeneous heat estimate.
Thus the cited contraction lemma applies if
$\|F\|_{L^2H^{-1/2}}<ce^{-CL}$, after changing universal constants.
Since
$\|F\|_{L^2(0,L;H^{-1/2})}=\nu^{-3/2}\|f\|_{L^2(0,S;H^{-1/2})}$,
the stated radius suffices. Persistence of smoothness and local existence
at the final time give lifespan strictly greater than $S$.
The finite-time inhomogeneous estimate is essential: small
$H^{-1/2}$ norm does not bound the homogeneous negative norm at low frequencies.
\end{proof}

\begin{proof}[Proof of Theorem~\ref{thm:Rmain}]
Fix $a,g$ and a positive radius in the indicated force norm. There are two cases. If $T^\nu_{\max,\R}(a,g)\le T$, take $f=g$. Otherwise choose $\delta>0$ so that the given smooth solution exists through $T+\delta$, and apply Theorem~\ref{thm:Rinsert}. For $s<s_q$, its force difference is smaller than the prescribed radius when $\eps$ is sufficiently small, its initial velocity is exactly $a$, and its lifespan is exactly $T$. This proves density for each fixed smooth initial velocity.

For zero initial velocity and $q=1$, Proposition~\ref{prop:Rcritical1} provides an open ball about zero disjoint from $\mathcal B^\R_{\nu,0,T}$ in $L^1_tH^{1/2}_x$. The inclusion $H^s\hookrightarrow H^{1/2}$ with norm at most one for $s\ge1/2$ gives the same obstruction in each stronger metric. For $q=2$, use Proposition~\ref{prop:Rcritical2} with $S=T$ and then $H^s\hookrightarrow H^{-1/2}$ for $s\ge-1/2$. These nonempty relative open balls prove non-density. The energy and earlier-history assertions are part of the gluing theorem.
\end{proof}

\subsection{Compact support, rapid decay, and completed force spaces}
Define the following two subclasses of $\mathcal F_{\R}$:
\begin{align*}
 \mathcal F_c&=C_c^\infty(\R^3\times(0,\infty);\R^3),\\
 \mathcal F_{\mathrm{rd}}
 &=\left\{f\in C^\infty(\R^3\times[0,\infty)):
 \sup_{x\in\R^3,\ t\ge0}(1+|x|+t)^N
       |\partial_x^\alpha\partial_t^jf(x,t)|<\infty
 \text{ for all }N,\alpha,j\right\}.
\end{align*}
Vector values are understood in the second line. The initial class corresponding to the rapid-decay formulation is $\mathcal S_\sigma=\mathcal S(\R^3;\R^3)\cap L^2_\sigma$. These are exactly the decay requirements relevant to \cite[equations (4)--(5)]{clay}; a common numerical bound on their seminorms is not imposed.

\begin{corollary}\label{cor:Rclasses}
Theorem~\ref{thm:Rmain} remains valid with $\mathcal F_{\R}$ replaced by $\mathcal F_c$ or $\mathcal F_{\mathrm{rd}}$. In particular it holds for every fixed $a\in\mathcal S_\sigma$ in the rapid-decay class, with the complete if-and-only-if classification at $a=0$.
\end{corollary}
\begin{proof}
The initial value is preserved exactly, and the force correction in Theorem~\ref{thm:Rinsert} is spacetime compact. Addition of that correction preserves each stated force class, including all rapid-decay bounds. The two-case density proof therefore stays within the chosen subclass. Each critical regular ball contains zero and remains a nonempty relative open ball in that subclass. The local framework only requires $a\in\mathcal X_{\R}$, so it applies in particular to Schwartz initial data. Nothing here requires the nonzero given velocity or its pressure to be spatially compact.
\end{proof}

The force spaces in the next proposition are natural for the energy
class. With $V=H^1(\R^3;\R^3)\cap L^2_\sigma$, a finite-energy velocity
belongs to $L^\infty_tL^2_\sigma\cap L^2_tV$. A full distribution in
$H^{-1}$ restricts to an element of $V'$; the latter records only its
action on divergence-free tests. This distinction is discussed in
\cite[Definition~1 and Remark~4]{berselli}. Our realization of
$\dot H^{-1}$ gives the pairing
\[
 |\ip{h}{z}|\le\norm{h}_{\dot H^{-1}}\norm{\nabla z}_2,
 \qquad
 \left|\int_0^T\ip{f}{u}\dd t\right|
 \le\norm{f}_{L^2_t\dot H^{-1}_x}\norm{\nabla u}_{L^2_{t,x}}.
\]
These inequalities follow by Fourier Cauchy--Schwarz, first for Schwartz
functions and then by completion. The inhomogeneous $H^{-1}$ norm instead
pairs with $H^1$, while $L^1_tL^2_x$ pairs with $L^\infty_tL^2_x$.

\begin{proposition}\label{prop:Renergy}
Fix $a\in\mathcal X_{\R}$ and $\nu,T>0$. Smooth compact forces with $T^\nu_{\max,\R}(a,f)\le T$ are dense in each full Bochner space $L^q(0,\infty;H^s(\R^3))$ for $q\in\{1,2\}$ and $s<s_q$. They are also dense in $L^2(0,\infty;\dot H^{-1}(\R^3))$.

For every given smooth solution in Theorem~\ref{thm:Rinsert}, one may simultaneously arrange
\[
 \norm{u_\eps-v}_{E_T}\longrightarrow0,\qquad
 \norm{g_\eps-g}_{L^1_tL^2_x}
 +\norm{g_\eps-g}_{L^2_tH^{-1}_x}
 +\norm{g_\eps-g}_{L^2_t\dot H^{-1}_x}\longrightarrow0.
\]
The homogeneous norm here is a norm of the compact difference; the given force itself need not belong to that homogeneous force space.
\end{proposition}
\begin{proof}
The standard density of Schwartz functions in $H^s$
\cite[Chapter~4, Section~1, Exercise~1]{taylor2011}, followed by smooth spatial
cutoffs, gives density of $C_c^\infty$ for every real $s$. For the homogeneous
space, approximate $|\xi|^{-1}\widehat h$ by smooth annular Fourier data and
then use the same cutoffs. The only additional low-frequency estimate is
\begin{equation}\label{eq:Rnegative-cutoff}
 \norm{k}_{\dot H^{-1}}^2\le
 C\norm{k}_1^2\int_{|\xi|<1}|\xi|^{-2}\dd\xi+\norm{k}_2^2
 \le C'\norm{k}_1^2+\norm{k}_2^2.
\end{equation}
Applied to the Schwartz cutoff tails, this proves spatial compact-smooth
density in $\dot H^{-1}$. The constructions preserve real vector fields.
For time dependence, Hunter~\cite[Proposition~6.29]{hunterpde} gives density
of finite sums $\sum_j\varphi_j(t)b_j$ with $\varphi_j\in C_c^\infty$ in
$L^q(I;X)$, $q<\infty$. Truncate $(0,\infty)$ to $[1/N,N]$ and approximate
each $b_j$ by a spatially compact smooth field. The resulting finite sums
are jointly smooth and compactly supported in space and positive time.

For the inhomogeneous density of forces producing breakdown, first choose such a smooth compact force within half the prescribed radius of the target. Corollary~\ref{cor:Rclasses} supplies a smooth compact force producing breakdown within the other half. This is density of a smooth subset in the completion, not a definition of classical breakdown for every rough force.

For the homogeneous gluing estimate, apply the Fourier scaling calculation of \eqref{eq:RnegativeScale} at $s=-1$. Explicitly, an amplitude $\eps^{-a}$ has spatial $\dot H^{-1}$ factor $\eps^{5/2-a}$, because
$\widehat{\eps^{-a}h(\,\cdot/\eps)}(\xi)=\eps^{3-a}\widehat h(\eps\xi)$.
The time $L^2$ factor is $\eps$. The source force has $a=3$, while the correction has $a=2$ and uniformly bounded compact rescaled profiles; their $\dot H^{-1}$ norms are uniformly finite by \eqref{eq:Rnegative-cutoff}. Hence
\[
 \norm{F_\eps}_{L^2_t\dot H^{-1}_x}\le C\eps^{1/2},\qquad
 \norm{H_\eps}_{L^2_t\dot H^{-1}_x}\le C\eps^{3/2}.
\]
Given a compact smooth force, if its lifespan is at most $T$ it already belongs to the required breakdown set. Otherwise glue the building block and use these estimates. Combining this relative approximation with the proved compact-smooth Bochner density establishes density in $L^2_t\dot H^{-1}_x$. Finally these estimates, \eqref{eq:REclose}, and the $q=1,s=0$ and $q=2,s=-1$ cases of the gluing estimates prove the simultaneous convergence for the same family.
\end{proof}

As on the torus, every $a\in\mathcal X_{\R}$ occurs as the initial velocity of a solution producing breakdown for some smooth force. Thus the projection of the initial data and forces producing breakdown onto $\mathcal X_{\R}$ is the whole space, with quantifier order $\forall a\,\exists f$. The preceding approximation concerns finite-energy whole-space solutions and uses one rescaled copy of $U$ in each perturbation.

\subsection{Cell averages on prescribed grids}\label{sec:grids}
The spaces used for energy estimates should be distinguished from the
additional regularity assumed in numerical error estimates. For example,
on a bounded no-slip computational domain one uses the solenoidal energy
space $H$ with zero normal trace and $V=H^1_0\cap H$; finite-dimensional
velocity and pressure spaces discretize this setting. Guermond, Minev and
Shen~\cite[Section~2 and Theorem~3.2]{gms} state these spaces and the
additional smoothness required for their error estimate. The initial
pressure is constrained by the equation and boundary conditions. These
bounded-domain conventions do not impose boundary conditions at infinity
in the whole-space result below.

To relate the continuous norms to a concrete numerical observation, let $\mathcal T_h$ be a complete uniform Cartesian grid of $\R^3$ with positive mesh widths. For a locally integrable vector field, define the cell-observation map with sequence codomain
\[
 A_h:L^1_{\mathrm{loc}}(\R^3;\R^3)\longrightarrow(\R^3)^{\mathcal T_h},
 \qquad (A_hz)_C=\frac1{|C|}\int_Cz(x)\dd x.
\]
Only coordinatewise equality in this codomain is needed. If a sequence norm is desired for $z\in L^2$, the natural volume-weighted norm obeys
$\sum_{C\in\mathcal T_h}|C|\,|(A_hz)_C|^2\le\norm{z}_2^2$
by Cauchy--Schwarz in each cell and summation. Each grid has infinitely many cells. The next result concerns a finite number of such grids and equality on every cell of each grid, at every time before blowup.

\begin{theorem}\label{thm:Rgrid}
Under the hypotheses on the given smooth solution of Theorem~\ref{thm:Rinsert}, fix a finite family of complete uniform Cartesian grids. The glued solutions may be chosen so that, for every grid in that family,
\[
 A_hu_\eps(t)=A_hv(t),\qquad A_hg_\eps(t)=A_hg(t)
 \qquad(0\le t<T),
\]
while $T^\nu_{\max,\R}(a,g_\eps)=T$ and the energy and force convergences in Proposition~\ref{prop:Renergy} hold. All differences are supported inside a ball contained in one cell of every grid, except for an optional spatially constant pressure gauge.
\end{theorem}
\begin{proof}
The union of faces in one grid is a closed, locally finite union of planes. For the fixed finite family their union is closed and has measure zero. Choose a point outside that union and a ball with closure inside one cell of each grid. Apply Theorem~\ref{thm:Rinsert} inside that ball, using the compact pressure representative. Put $\delta u=u_\eps-v$, $\delta p=p_\eps-\pi=P_\eps$, and $\delta g=g_\eps-g$.

Fix one containing cell $C$. For every $t<T$, $\delta u$ is smooth, divergence free and compactly supported in its interior. For component $j$,
\[
 \delta u_j=\nabla\cdot(x_j\delta u),\qquad
 \int_C\delta u_j\dd x=\int_{\partial C}x_j\delta u\cdot n\dd S=0.
\]
Subtract the two momentum equations and integrate to obtain
\begin{equation}\label{eq:gridforce}
 \int_C\delta g\dd x
 =\frac{\dd}{\dd t}\int_C\delta u\dd x
 +\int_{\partial C}\bigl(u_\eps\otimes u_\eps-v\otimes v
                         +\delta p\,I-\nu\nabla\delta u\bigr)n\dd S.
\end{equation}
The time derivative is zero by the preceding identity. All differences vanish in a neighborhood of the cell boundary, so the surface term is zero. A spatially constant pressure gauge adds a multiple of $\int_{\partial C}n\dd S=0$. The averages therefore agree in $C$; in every other cell they agree by support. The same construction fits the containing cell of each grid, and the convergence estimates are unaffected by this fixed upper bound on $\eps$.
\end{proof}

A deterministic procedure using only the initial velocity cell averages
and the force cell averages for $0\le t<T$ receives identical data in the
two cases. Other observations, such as point values or pressure, may
distinguish them. The perturbation depends on the fixed grid family and
develops larger amplitudes and derivatives as its support shrinks. The
result therefore concerns the information in the prescribed averages,
rather than convergence under refinement for one fixed smooth problem.

\section*{Declaration of generative AI use}
GPT Astra (OpenAI) assisted with literature review, development and checking
of mathematical arguments, drafting, and LaTeX preparation. Codex assisted
with manuscript consolidation and revision. Responsibility for the final
content remains with the authors.

\end{document}